\documentclass[12pt]{amsart} 
\usepackage{amssymb}
 \calclayout
\theoremstyle{plain}
\newtheorem{thm}{Theorem}[section]

\newtheorem{lem}{Lemma}[section]

\newtheorem{prop}{Proposition}[section]

\newtheorem{defn}{Definition}[section]
\newtheorem{ex}[thm]{Example}
\newtheorem{rem}{Remark}[section]

\newcommand{\begintheorem}{\addtocounter{equation}{1}\begin{theorem}}
\newcommand{\beginlemma}{\addtocounter{equation}{1}\begin{lemma}}
\newcommand{\beginproposition}{\addtocounter{equation}{1}\begin{proposition}}
\newcommand{\begindefinition}{\addtocounter{equation}{1}\begin{definition}}
\newcommand{\begincorollary}{\addtocounter{equation}{1}\begin{corollary}}

\usepackage[bookmarksnumbered, colorlinks, plainpages]{hyperref}
\hypersetup{colorlinks=true,linkcolor=red, anchorcolor=green, citecolor=cyan, urlcolor=red, filecolor=magenta, pdftoolbar=true}
\begin{document} 
%
%
%
%
%
%
%
%
%

\title[ Weakly $ p $-sequentially continuous differentiable mappings ]
{Weakly $ p $-sequentially continuous differentiable mappings}
\author {M.\ Alikhani.}
\address{Department of Mathematics, University of Isfahan}
\email{m2020alikhani@ yahoo.com}

\begin{abstract} In this article, we introduce the notions uniformly $ p $-convergent sets and weakly $ p$-sequentially continuous differentiable mappings.\ Then,  we obtain a sufficient condition for those Banach spaces which either contains no copy of $ \ell_{1} $ or have the $ p $-Schur property.
 Finally, we give a characterization for the weakly $ p $-sequentially continuous differentiable mappings.\
\end{abstract}
\maketitle
\section{Introduction:} 
Suppose that $ X $ and $ Y $ are real Banach spaces and $ U $ is an open subset of $ X. $\ A subset $ B \subset U $ is $ U $-bounded, if it is bounded and the distance 
between $ B$ and the boundary $ $ of $ U$ is strictly positive \cite{ce3}.\
Let $ C^{1u}(U, Y )$ be the space of all differentiable mappings
$ f : U \rightarrow Y $ whose derivative $ f^{\prime} : U \rightarrow L(X, Y ) $ is uniformly continuous on
$ U$-bounded subsets of $ U. $\ In recent years, several authors have studied properties of $ C^{1u} $ mappings between Banach spaces.\  For more information in this area, we refer the reader to \cite{dh, DM,h} and references therein.\ Recall that, a mapping $ f : U \rightarrow Y $ is compact if it takes $ U$-bounded subsets of $ U $ into relatively compact subsets of $ Y. $\
Aron {\rm (\cite[Theorem 1.7]{A})}, proved that a holomorphic linear function $ f : X \rightarrow \mathbb{C} $
is compact on bounded subsets if and only if its derivative
$ f^{\prime}:X\rightarrow X^{\ast} $ is a compact mapping.\ Bombal et al. \cite{ce},
for a holomorphic function $  f$ of bounded type on a complex Banach space $ X, $  proverd
that its derivative $ f^{\prime} :X\rightarrow X^{\ast}$ takes bounded sets into certain families of sets if and only
if $ f $ may be factored in the form $ f = g \circ S, $ where $  S$ is in some associated operator ideal,
and $ g $ is a holomorphic function of bounded type.\ Following this, Cilia et al. {\rm (\cite[Theorem 3.1]{ce4})}, proved that if $ f^{\prime} $ is uniformly continuous on bounded sets and with values in the space $ K(X, Y )$ of compact operators from $ X$ into $ Y, $ then $ f^{\prime} $ is compact if and only if $ f$ is weakly uniformly
continuous on bounded sets.\
It is well known {\rm (\cite[Proposition 3.2]{gg})}, that a bounded linear operator $ T:X\rightarrow Y $ between Banach spaces is completely
continuous if and only if its adjoint $ T^{\ast}$ takes bounded subsets of $ Y^{\ast} $ into
uniformly completely continuous subsets, often called $ (L) $-subsets, of $ X^{\ast} .$\ Let us recall from \cite{spd}, that 
a subset $ K \subseteq L(X, Y ) $ is called uniformly completely
continuous, if for every weakly null sequence $ (x_{n})_{n} $ in $ X, $ it follows:
$$\lim_{n\rightarrow \infty}\sup_{T \in K}\parallel T(x_{n}) \parallel =0. $$
Motivated  by this work, Cilia and Guti\'errez {\rm (\cite[Theorem 2.1]{ce2})}  gave a similar
result for differentiable mappings.\ In fact, they showed that if $ f \in C^{1u}(U, Y ) ,$ then $ f$ takes weakly Cauchy  $ U$-bounded
sequences into convergent sequences if and only if $ f^{\prime} $ takes $ U$-bounded Rosenthal subsets
of $ U$ into uniformly completely continuous subsets of $ X^{\ast} .$\\ 
Recently, Li et al.\cite{ccl1} generalized the concepts $ (L) $ sets and $ (V) $ sets to
the $ p $-$ (V ) $ sets for $ 1\leq p \leq \infty. $\ It is well known {\rm (\cite[Theorem 14]{g12})} that, a bounded linear operator $T:X\rightarrow Y $ between Banach spaces is $ p$-convergent
if and only if its adjoint $ T^{\ast} : Y^{\ast}\rightarrow X^{\ast}$ takes bounded subsets of $ Y^{\ast} $ into $p $-$ (V) $ subsets of $ X^{\ast} .$\\ 
Inspired by the above works, we introduce the notions uniformly $ p $-convergent sets and weakly $ p$-sequentially continuous mappings.\ In the sequel, we obtain a sufficient condition for those Banach spaces which either contains no copy of $ \ell_{1} $ or have the $ p $-Schur property.\
Finally, 
we show that 
if $ U $ is an open convex subset of $ X $ and $ f \in C^{1u}(U, Y ) ,$ then $ f $ takes $ U$-bounded and
weakly $ p $-Cauchy 
sequences into norm convergent sequences if and only if $ f^{\prime} $ takes $ U$-bounded and weakly $ p $-precompact subsets of $ U$ into uniformly $ p $-convergent
subsets of $ L(X, Y ). $\ 
\section{Notions and Definitions:} 
Throughout this paper $ 1\leq p < \infty $ and $ 1\leq p < q\leq \infty,$ except for the cases where we consider other assumptions.\ Also, we suppose
$ X $ and $ Y$ will denote real Banach spaces, $ U \subseteq X $ will be an open convex subset,
$p^{\ast}$ is the H$\ddot{\mathrm{o}}$lder conjugate of $p.$\ The word `operator' will always mean
bounded linear operator. For any Banach space X, the dual space of bounded
linear functionals on X will be denoted by $ X^{\ast} .$\ Also, we denote the closed unit ball of $ X$ by $ B_{X}. $\ For a bounded linear operator $ T : X \rightarrow Y, $ the adjoint of the operator $ T $
is denoted by $ T^{\ast}. $\ The space of all bounded linear operators and weakly compact operators from $ X$ to $ Y$ will be denoted by $ L(X,Y ) $ and
$ W(X,Y ), $ respectively.\ If $ X_{1},\cdot\cdot\cdot, X_{k} $ are Banach spaces, the space of all $ k$-linear (continuous) mappings
from $ X_{1}\times\cdot\cdot\cdot\times X_{k} $ into $ Y$ is denoted by $ L^{k}(X_{1},\cdot\cdot\cdot,X_{k}, Y) .$\ A mapping $ P:X\rightarrow Y $ is a $ k$-homogeneous (continuous) polynomial if there is a $ k$-linear mapping
$ A : X\times\cdot\cdot\cdot\times X \rightarrow Y $ such that $ P(x) = A(x, . . . , x) $ for all $ x \in X. $\
The space of all $ k$-homogeneous continuous polynomials from $ X $ into $ Y $ is denoted by $ P(^{k}X,Y).$\
Note that correspondent each $ P\in P(^{k}X , Y), $ we can associate a unique symmetric $ k $-linear mapping $ \hat{P}: X\times\cdot\cdot\cdot\times X\rightarrow Y $ so that $ P(x)=\hat{P}(x,\cdot\cdot\cdot,x). $\
A mapping $ f:X\rightarrow Y $ is holomorphic if, for each $ x\in X $ there are $ r>0 $ and a sequence $ (P_{k}) $ of polynomials, with
$P_{k}\in P(^{k}X,Y)$ such that $ f$ may be given by its Taylor series expansion around $ x: $
$$ f(y)=\sum_{k=0} ^{\infty}P_{k}(y-x),$$
uniformly for $ \parallel y-x\parallel < r. $\ We use the notation $ \mathcal{H} (X,Y)$ for the space of all holomorphic mappings from $ X$ into $ Y. $\ A mapping $ f \in \mathcal{H} (X,Y)$ is called bounded type, if $ f $ is bounded on bounded sets.\ The space of all bounded type mappings from $ X$ into $ Y $ is denoted by $ \mathcal{H}_{b} (X,Y).$ 
If the range space is omitted in the above notations, it is understood to be the real field $ \mathbb{R} $ for instance, $ \mathcal{H}(X):= \mathcal{H} (X,\mathbb{R}) ,$
$ P(^{k}X):=P(^{k}X, \mathbb{R}) $ and  $ \mathcal{H}_{b} (X):= \mathcal{H}_{b} (X,\mathbb{R}). $\ 
Given $ x, y \in X, $ we write $ I(x, y) $ for the
segment with bounds x and y, that is,
$$ I(x, y)=\lbrace x+\lambda (y-x): 0\leq \lambda \leq 1 \rbrace .$$\
An operator is completely
continuous, if it takes weakly convergent sequences into norm convergent sequences.\ The subspace of all such operators
is denoted by $ CC(X, Y ).$\
Let us recall from \cite{djt}, that a sequence $ (x_{n})_{n} $ in $ X $ is called weakly $ p $-summable, if $ (x^{\ast}(x_{n}))_{n} \in \ell_{p}$ for each $ x^{\ast}\in X^{\ast} $.\ A sequence $(x_{n})_{n}$ in $ X $ is called weakly $ p $-convergent to $ x\in X,$ if $ (x_{n} - x)_{n}$ is weakly $ p$-summable.\ A sequence $(x_{n})_{n}$ in $ X $ is called weakly
$ p $-Cauchy, if $ (x_{m_{k}}-x_{n_{k}})_{k}$ is weakly $ p $-summable for any increasing sequences $ (m_{k})_{k} $
and $ (n_{k})_{k} $ of positive integers \cite{ccl}.\ Note that, every weakly $ p $-convergent sequence is
weakly $ p$-Cauchy, and the weakly $ \infty $-Cauchy sequences are precisely the weakly
Cauchy sequences.\ 
 An operator $ T : X \rightarrow Y $ is called $ p$-convergent, if $ T $ maps weakly $ p $-summable
sequences into norm null sequences \cite{cs}.\ The class of all $ p$-convergent operators from $ X$
into $ Y $ is denoted by $ C_{p}(X,Y) .$\ A  Banach space $ X $ has the $ p$-Schur property, if the
identity operator on $ X$ is $ p $-convergent.\ A bounded subset $ K$ of
$ X^{\ast} $ is a $ p $-$ (V ) $ set, if for every weakly $ p $-summable sequence
$ (x_{n})_{n} $ in $ X, $ it follows:
$$ \lim _{n} \sup_{x^{\ast}\in K} \Vert x^{\ast}(x_{n})\Vert =0.$$
A subset $ K$ of $ X $ is called weakly $ p$-precompact,
if every sequence from $ K $ has a weakly $ p$-Cauchy subsequence \cite{ccl}.\ Note that the
weakly $ \infty $-precompact sets are precisely the weakly precompact sets or Rosenthal sets.\
We refer the reader for undefined terminologies to the
classical references \cite{AlbKal, djt}.\
\section{Main results:}
In this section, we obtain a characterization for
the weakly p-sequentially continuous mappings.\ Namely,
we show
that $ f:U\rightarrow Y $ is weakly $ p $-sequentially continuous if and only if $ f^{\prime} $ takes $ U$-bounded and weakly $ p $-precompact subsets of $ U$ into uniformly $ p $-convergent
subsets of $ L(X, Y ). $\ 

\begin{defn}\label{d1}
Let $ K \subset L(X, Y ). $\ We say that $ K $ is a uniformly $ p $-convergent set, if for every weakly $ p $-summable sequence
$ (x_{n})_{n} $ in $ X, $ it follows:
$$ \lim _{n} \sup_{T \in K} \Vert T(x_{n})\Vert =0.$$
\end{defn}
As an immediate consequence of the Definition \ref{d1}, one can conclude
that the following elementary results.
\begin{prop}\label{p1} 
$ \rm{(i)} $ Every subset of an uniformly $ p $-convergent set in $ L(X,Y) $ is uniformly $ p $-convergent.\\
$ \rm{(ii)} $ Absolutely closed convex hull of an uniformly $ p $-convergent set in $ L(X,Y) $ is uniformly $ p $-convergent.\\
$ \rm{(iii)} $ If $ K_{1},\cdot \cdot\cdot, K_{n} $ are uniformly $ p $-convergent sets in $ L(X,Y), $ then $ \displaystyle\bigcup_{i=1}^{n} K_{i}$ and $ \displaystyle \sum_{i=1}^{n} K_{i} $ are uniformly $ p $-convergent sets in $ L(X,Y). $\
\end{prop}
\begin{rem}
$ \rm{(i)} $ If $ K\subset X^{\ast} ,$ then the definition of uniformly $ p $-convergent sets coincide
with the definition of $ p$-$ (V) $ sets.\
Also, every
uniformly $ q $-convergnt subset of $L(X,Y) $ is 
uniformly $ p $-convergnt, whenever $1\leq p < q \leq \infty. $\ Also, it would be interesting to obtain conditions under which every uniformly $ p $-convergent set in $L(X,Y) $ is 
uniformly $ q $-convergnt.\ In my opinion, this is a very interesting, but it is
a difficult question?\ In particular, we
answer this question.\ Indeed, we obtain a characterization for those Banach
spaces in which uniformly $ p $-convergent sets in $ X^{\ast} $ are uniformly $ q $-convergent (see Definition $ \rm{2.1} $ and Theorem
$ \rm{2.4} $ in \cite{afz}).\\
$ \rm{(ii)} $ Suppose that $ P\in P(^{k}X) .$\ It is easy to verify that the equivalence $ (a) \Leftrightarrow (e) $ in {\rm (\cite[Corollary 9 ]{ce})}, shows that the operator 
$ \overline{P}:X\rightarrow P(^{k}X) $ given by $ \overline{P}(x)(y) =\hat{P}(x,\cdot\cdot\cdot,x ,y)~~~(x,y \in X)$ belongs to $ C_{p} $
if and only if the derivative $ P^{\prime} \in P(^{k}X, X^{\ast}) $ defined by $ P^{\prime}(x)(y)=k \hat{P} (x,\cdot\cdot\cdot,x ,y)$ takes bounded sets of $ X $ into uniformly $ p $-convergent sets in $ X^{\ast} .$\\  
$ \rm{(iii)} $ It is well known {\rm (\cite[Theorem 11]{ce})}, that for a holomorphic function $ f $ of bounded type on a complex Banach space $ X, $ 
its derivative $ f^{\prime} :X\rightarrow X^{\ast}$ 
takes bounded sets into uniformly $ p $-convergent sets in $ X^{\ast} $ if and only
if $ f $ may be factored in the form $ f=g \circ S $ where $ S $ is $ p $-convergent
and $ g\in \mathcal{H}_{b} (X).$
\end{rem}

The following example show that, there exists a uniformly $ p $-convergnt subset of $ L(\ell_{2},Y) $ so that 
it is not uniformly $ q $-convergnt.\
\begin{ex} Let $ X=\ell_{2} $ and $ Y$ be an arbitrary Banach space.\
Since $ \ell_{2} $ does not have $ 2 $-Schur property, $ B_{\ell_{2}} $ is not a $ 2$-$ (V ) $ set in $ \ell_{2} $.\ Therefore $B_{\ell_{2}} $ 
is not uniformly $ 2 $-convergnt subsets of $ \ell_{2} .$\
On the other hand, $ \ell_{2} $ contains no copy of $ c_{0} .$\ Therefore, $ \ell_{2} $ has the $ 1 $-Schur property (see \cite{dm}).\ Hence,
$ B_{\ell_{2}} $ is a $ 1 $-$ (V ) $ set and so, $ B_{\ell_{2}} $ is a uniformly $ 1 $-convergnt subsets of $ \ell_{2} .$\ Now, let $ 0\neq y_{0}\in B_{Y} $ and
$ S : \mathbb{R}\rightarrow Y $ be the operator
given by $ S(\lambda) := \lambda y_{0}~(\lambda \in \mathbb{R}). $\ Define an operator $ T : \ell_{2}\rightarrow L(\ell_{2}, Y ) $ by
$T (\phi)(h) := \phi(h) y_{0}, $ for $ \phi \in \ell_{2}^{\ast}=\ell_{2} $ and $ h \in \ell_{2}. $\ Then
$$ \parallel T(\phi) \parallel=\sup_{h\in B_{\ell_{2}}}\parallel T(\phi) (h) \parallel=\sup_{h\in B_{\ell_{2}}}\parallel \phi(h) y_{0} \parallel=\parallel \phi \parallel .$$ 
Since, uniformly $ p $-convergent sets are stable under isometry,
there exists a uniformly $ 1 $-convergnt subset of $ L(\ell_{2},Y) $ so that 
it is not uniformly $ 2 $-convergnt.\
\end{ex}
Recall that a sequence $ (x_{n}) \subset U $ is $ U $-bounded, if the set $ \lbrace x_{n}: n \in \mathbb{N}\rbrace $ is $ U $-bounded.\ A
mapping $ f : U \rightarrow Y $ is weakly sequentially continuous, if it takes $ U $-bounded and weakly Cauchy sequences of $ U $ into norm convergent sequences in $ Y. $\ The space of all such mappings is denoted by $ C_{wsc}(U, Y )$ \cite{ce2}.
\begin{defn} \label{d2} Let $ U$ be an open convex subset of $ X$ and $ 1 \leq p \leq \infty. $\
We say that $ f: U\rightarrow Y $ is a weakly $ p $-sequentially continuous map, if it takes weakly $ p $-Cauchy and $ U $-bounded sequences of $ U $ into norm convergent
sequences in $ Y. $\ The space of all such mappings is denoted by $ C_{wsc}^{p}(U, Y ). $
\end{defn}
\begin{rem}
It is easy to show that the weakly $ \infty $-sequentially continuous maps are precisely the weakly sequentially continuous.\
Also, it is clear that
for all $ X, Y , $ and for every open subset $ U \subseteq X, $ we have $ C^{q}_{wsc} (U,Y)\subseteq C^{p}_{wsc} (U,Y), $ whenever $ 1\leq p<q \leq \infty. $\ But, we do not have any example of a mapping $ f \in C^{1u}(U, Y ) \cap C^{p}_{wsc}(U, Y ) $ which does not belong to $ C^{q}_{wsc}(U, Y ). $\
Also, it would be interesting to obtain conditions under which every mapping in the space $ C^{1u}(U, Y ) \cap C^{p}_{wsc}(U, Y ) $ takes $ U $-bounded and weakly 
$ q $-Cauchy sequences into norm convergent sequences.\ In my opinion, this is a very interesting, but it is a difficult question?\

\end{rem}
\begin{prop}\label{p2}
Suppose that $ U$ is an open subset of $ X. $\ If $ f $ is compact and takes $ U $-bounded weakly $ p $-Cauchy sequences into weakly Cauchy sequences, then $ f \in C_{wsc}^{p}(U, Y ). $ 
\end{prop}
\begin{proof}
Let $ (x_{n})_{n} $ be a weakly $ p $-Cauchy and $ U$-bounded sequence.\ Then $ ( f (x_{n}))_{n} $ is weakly Cauchy so that, $ \lbrace f(x_{n}) :n\in\mathbb{N} \rbrace $ is a
relatively compact set.\ It is easy to show that this implies the norm convergence of $ ( f (x_{n}))_{n} .$\ Hence, $ f \in C_{wsc}^{p}(U, Y ). $ 
\end{proof}

\begin{prop}\label{p3}
Let $ U $ be an open convex subset of $ X$ and let $ f \in C^{1u}(U, Y ). $\ If $ f^{\prime} \in C^{p}_{wsc} (U, C_{p}(X,Y)),$ then $ f\in C_{wsc}^{p} (U,Y).$
\end{prop}
\begin{proof}
Let $ (x_{n})_{n} $ be a $ U $-bounded weakly $ p $-Cauchy sequence.\ By the Mean Value Theorem {\rm (\cite[Theorem 6.4]{ch})}, we have
$$ \parallel f(x_{n})-f(x_{m}) \parallel ~\leq ~\parallel f^{\prime}(c_{n,m}) (x_{n}-x_{m})\parallel$$
for some $ c_{n,m} \in I(x_{n}, x_{m}). $\ Since the sequence $ (c_{n,m}) $ is $ U$-bounded and weakly $ p $-Cauchy, by the hypothesis the sequence $ (f^{\prime}(c_{n,m})) $ norm
converges
to some $ T \in C_{p}(X, Y ). $\ Therefore we have:\\ $$ \displaystyle\lim_{n,m\rightarrow \infty} \parallel f^{\prime}(c_{n,m}) (x_{n}-x_{m})-T(x_{n}-x_{m})\parallel+T(x_{n}-x_{m}) \parallel $$\\ $$ \leq \lim_{n,m\rightarrow \infty} \parallel f^{\prime}(c_{n,m}) (x_{n}-x_{m})-T(x_{n}-x_{m})\parallel+\lim_{n,m\rightarrow\infty}\parallel T(x_{n}-x_{m}) \parallel=0.$$
Hence,
$ \displaystyle\lim_{n,m\rightarrow \infty}\parallel f^{\prime}(c_{n,m}) (x_{n}-x_{m})\parallel=0.$
\end{proof}
\begin{lem}\label{l1}
If $ X $ does not have the $ p $-Schur property, then there is a non $ p $-convergent operator $ T:X\rightarrow L_{\infty}[0,1].$
\end{lem}
\begin{proof}
If X is separable, then by {\rm (\cite[Theorem 2.5.7]{AlbKal})} $ X $ embeds isometrically into $ \ell_{\infty}.$\ 
Hence,
there is an into isometry $ T:X\rightarrow \ell_{\infty} $ which is obviously not $ p $-convergent.
If $ X $ is not separable, it is easy to show that it has a separable subspace $ Z $ without the $ p $-Schur property.\ Hence, there is
an into isometry $ T_{1} \in L(Z, \ell_{\infty}) $ and thus using the fact that $ \ell_{\infty} $ is injective we can
find a bounded linear operator $ T:X\rightarrow \ell_{\infty} $ which extends the operator $ T_{1} .$\ It is clear that $ T $ is not a $ p $-convergent operator.\ Since $ \ell_{\infty} $ is isomorphic to $ L_{\infty}[0, 1] $ (see {\rm (\cite[Theorem 4.3.10]{AlbKal})} ), we obtain a non $p $-convergent operator $ T :X\rightarrow L_{\infty}[0, 1]. $\
\end{proof}
The proof of following lemma with some minor modifications is similar to the
proof of {\rm (\cite[Theorem 2.3]{AHV})}.\ Hence, we omit its proof. 
\begin{lem}\label{l2}
A polynomial is weakly $ p $-sequentially continuous if and only if it is $ p$-convergent.
\end{lem}
A Banach space $ X$ is
said to have the Dunford-Pettis-property of order $ p $ (or briefly, $ X$ has the $ (DPP_{P}) $), 
if the inclusion $ W(X, Y ) \subseteq C_{p}(X, Y ) $ holds for all Banach spaces $ Y  $ \cite{cs}.
\begin{thm}\label{t4}
Let $ U $ be an open convex subset of $ X. $\ If for all $ k \in \mathbb{N} ~(k\geq2),$ every Banach space $ Y, $ and every weakly $ p $-sequentially continuous polynomial $ P \in P(^{k}X, Y ), $ the derivative
$ P^{\prime} \in P( ^{k-1}X,L(X, Y )) $ is weakly $ p $-sequentially continuous, then either $ X $ has the $ p $-Schur property or it
contains no copy of $ \ell_{1}. $
\end{thm}
\begin{proof}
Assume on the contrary that $ X $ contains a copy of $ \ell_{1}$ and does not have the $ p $-Schur property.\ Therefore, by Corollary 4.16 and Theorem 2.17 of \cite{djt}, there is a $ p $-convergent operator from $ X $ onto $ \ell_{2} .$\ Since $ \ell_{2} $ is isometrically
isomorphic to $ L_{2}[0, 1] $, we obtain a surjective operator $ S \in C_{p}(X, L_{2}[0, 1]) .$\ On the other hand, by Lemma \ref{l1}, there is a non $ p $-convergent operator $ T:X\rightarrow L_{\infty}[0,1].$\ Now, let $ k\geq 2 $ be the integer.\ Define a $ k $-linear mapping $ L:X^{k} \rightarrow L_{2}[0,1]$ by\\

$ L(x_{1},\cdot\cdot\cdot,x_{k}) =\frac{1}{k}[S(x_{1})T(x_{2})\cdot\cdot\cdot T(x_{k}) +\cdot\cdot\cdot +T(x_{1})T(x_{2})\cdot\cdot\cdot T(x_{k-1})S(x_{k})]$ for $ x_{1},\cdot\cdot\cdot, x_{k} \in X.$\
Obviously, $ L$ is well defined, symmetric, and continuous.\ For each $ x\in X, $  consider the associated polynomial
$ P(x)=S(x) T(x)^{k-1}.$\ We claim that $ P $ is weakly $ p $-sequentially continuous.\ For this purpose,
suppose that $ (x_{n}-x)_{n} $ is a 
weakly $p $-summable sequence in $ B_{X}. $\ Hence, we have:
\begin{eqnarray*}
 \parallel P(x_{n}) -P(x)\parallel= \parallel S(x_{n}) T(x_{n})^{k-1} - S(x) T(x)^{k-1} \parallel &\leq&\\ \parallel S(x_{n}) T(x_{n})^{k-1} - S(x) T(x_{n})^{k-1} \parallel +\parallel S(x) T(x_{n})^{k-1} - S(x) T(x)^{k-1} \parallel &= &\\ ( \int_{0}^{1} \vert S(x_{n})(t) -S(x)(t) \vert^{2} \vert T(x_{n}) (t)\vert^{2k-2} dt)^{\frac{1}{2}} &+ &\\ ( \int_{0}^{1} \vert S(x)(t) \vert^{2} \vert T(x_{n}) (t)^{k-1} -T(x)(t)^{k-1}\vert^{2} dt )^{\frac{1}{2}} &\leq &\\ \parallel T \parallel^{k-1} ( \int _{0}^{1} \vert S(x_{n}-x)(t)\vert^{2} dt)^{\frac{1}{2}}+ \langle S(x)^{2}, [T(x_{n}) (t)^{k-1} -T(x)(t)^{k-1}]^{2} \rangle^{\frac{1}{2}} &= &\\ \parallel T \parallel^{k-1} \Vert S(x_{n}-x)\Vert_{L_{2}}+\langle S(x)^{2}, [T(x_{n}) (t)^{k-1} -T(x)(t)^{k-1}]^{2} \rangle^{\frac{1}{2}}.
\end{eqnarray*}
Since $ S \in C_{p}(X, L_{2}[0, 1]) ,$ we have 
$\lim_{n\rightarrow \infty}\parallel S(x_{n}-x) \parallel_{L_{2}}=0. $\
Now, for $ m\in \mathbb{N} ~(m\geq 2), $ we consider the polynomial $ \varphi:L_{\infty} [0,1]\rightarrow L_{\infty} [0,1]$ defined by $ \varphi(g):=g^{m}. $\ Since $ L_{\infty} [0,1] $ has the $ (DPP_{p}) ,$
this polynomial takes weakly $ p $-convergent sequences into norm convergent sequences and so, the sequence $ ([T (x_{n})^{k-1}-T (x)^{k-1}]^{2})_{n} $ is weakly null in $ L_{\infty} [0,1] . $\ On the other hand, $ S(x)^{2}\in L_{1}[0,1] .$\ Hence, 
$$ \lim_{n\rightarrow \infty} \langle S(x)^{2}, [T(x_{n}) (t)^{k-1} -T(x)(t)^{k-1}]^{2} \rangle^{\frac{1}{2}} =0.$$
Therefore, $ \displaystyle \lim_{n\rightarrow \infty} \parallel P(x_{n}) -P(x)\parallel=0 $ and so,
Lemma \ref{l2} implies that the polynomial
$ P $ is weakly $ p $-sequentially continuous.\ By using the same
argument as in the proof of 
{\rm (\cite[Theorem 2.9]{ce2})}, $ P^{\prime} $ is not weakly $ p $-sequentially continuous, which is a contradiction.
\end{proof}

By using the same
argument as in the proof of {\rm (\cite[Theorem 2.1]{ce2})}, we obtain the following similar
result.
\begin{thm}\label{t1}
Let $ U\subseteq X $ be an open convex subset and let $ f \in C^{1u}(U, Y ). $\ Then the following assertions are equivalent:\\
$ \rm{(i)} $ $ f\in C_{wsc}^{p} (U,Y);$\\
$ \rm{(ii)} $ For every $ U $-bounded and weakly $ p $-Cauchy sequence $ (x_{n}) $ and every weakly $ p $-Cauchy sequence $ (h_{n}) \subset X, $ the sequence $ (f^{\prime}(x_{n})(h_{n}))_{n} $ norm converges
in $ Y; $\\
$ \rm{(iii)} $ For every $ U $-bounded weakly $ p $-Cauchy sequence $ (x_{n})_{n} $ and every weakly $ p $-summable sequence $ (h_{n})_{n} \subset X, $ we have
$$\lim_{n}\sup_{m}\parallel f^{\prime}(x_{m}) (h_{n}) \parallel =0;$$\
$ \rm{(iv)} $ For every $ U$-bounded weakly $ p $-Cauchy sequence $ (x_{n})_{n} $ and every weakly $ p $-summable sequence $ (h_{n})_{n} \subset X, $ we have
$$\lim_{n} f^{\prime}(x_{n}) (h_{n}) =0;$$
$ \rm{(v)} $ $ f^{\prime} $ takes $ U$-bounded and weakly $ p $-precompact subsets of $ U$ into uniformly $ p $-convergent subsets of $ L(X, Y ). $
\end{thm}
\begin{proof}
(i) $ \Rightarrow $ (ii) Suppose that $ (x_{n})_{n} $ is a $ U $-bounded and weakly $ p $-Cauchy sequence.\ Let $ (h_{n})_{n} $ be a weakly $ p $-Cauchy sequence in $ X. $\ Without loss
of generality, we assume that $ (h_{n})_{n}$ is bounded.\ Consider $ B := \lbrace x_{n}: n \in \mathbb{N} \rbrace $ and let $ d := min\lbrace 1, dist(B, \partial U)\rbrace. $\ It is easy to show that the set
$$ B^{\prime}:=B+\frac{d}{2}B_{X}\subset U$$
is also $ U $-bounded.\ Since $ f \in C^{1u}(U, Y ),$ $ f^{\prime} $ is uniformly continuous on $ B^{\prime}.$\ Hence, for given $ \varepsilon >0, $ there exists $ 0 <\delta <\frac{d}{4} $ such that if $t_{1} ,t_{2}\in B^{\prime}$
satisfy $ \parallel t_{1} -t_{2}\parallel <2\delta,$ then
\begin{equation}\label{eq}
\parallel f^{\prime}(t_{1})-f^{\prime}(t_{2}) \parallel <\frac{\varepsilon}{4}.
\end{equation}
If $ c\in I(x_{n},x_{n}+\delta h_{n}) $ for some $ n\in \mathbb{N}, $ then
$$ \parallel c-x_{n} \parallel \leq \delta \parallel h_{n} \parallel <\delta <2\delta <\frac{d}{2},$$
and so,
$$c=x_{n} +(c-x_{n})\in B^{\prime}=B+\frac{d}{2}B_{X} $$
As an immediate consequence of the Mean Value Theorem {\rm (\cite[Theorem 6.4]{ch})}, and formula (\ref{eq}), we obtain
\begin{flushleft}
$ \parallel f^{\prime}(x_{n}) ( \delta h_{n})-f(x_{n}+
\delta h_{n})+f(x_{n})\parallel $
\end{flushleft}
\begin{center}
$ \leq \displaystyle\sup_{c\in I(x_{n},x_{n}+\delta h_{n})} \parallel f^{\prime} (c) - f^{\prime} (x_{n}) \parallel \parallel\delta h_{n}\parallel
\leq\frac{\varepsilon \delta}{4} .$
\end{center}
Similary,
\begin{flushleft}
$ \parallel f(x_{m}+\delta h_{m})-f(x_{m})- f^{\prime}(x_{m}) ( h_{m})\parallel $
\end{flushleft}
\begin{center}
$
\leq \displaystyle\sup_{c\in I(x_{m},x_{m}+\delta h_{m})} \parallel f^{\prime} (c) - f^{\prime} (x_{m}) \parallel \parallel\delta h_{m}\parallel
\leq\frac{\varepsilon \delta}{4}.\
$
\end{center}
In the other word, the sequences $ (x_{n}+\delta h_{n})_{n} $ and $ (x_{n})_{n} $ are $ U$-bounded and weakly $ p $-Cauchy in $ U. $\ Hence, by the hypothesis the
sequences $ ( f (x_{n} +\delta h_{n}))_{n} $ and $ ( f (x_{n}))_{n} $ are norm convergent in $ Y. $\ Hence, we can find $ n_{0}\in \mathbb{N} $ so that for $ n,m > n_{0} :$
\begin{flushleft}
$ \parallel f(x_{n}+\delta h_{n})-f(x_{m}+\delta h_{m})\parallel <\frac{\varepsilon \delta}{4},~~~~~~~~~~~~~ \parallel f(x_{n})-f(x_{m})\parallel <\frac{\varepsilon \delta}{4} $
\end{flushleft}
So, for $ n,m > n_{0} ,$ we have\
$$\parallel f^{\prime}(x_{n}) (h_{n}) - f^{\prime}(x_{m}) (h_{m})\parallel< \varepsilon.$$\
(ii) $ \Rightarrow $ (iii) Let $ (x_{n})_{n} $ be a $ U $-bounded weakly $ p $-Cauchy sequence and let $ (h_{n})_{n} $ be a weakly $ p $-summable sequence in $ X. $\ By the part $ \rm{(ii)} $, for every $ h \in X, $ the set $ \lbrace f^{\prime} (x_{n})(h) :n\in \mathbb{N}\rbrace$ is bounded in $ Y. $\ On the other hand, there exists a subsequence $ (x_{m_{k}}) _{k}$ of $ (x_{m})_{m} $ in $ U $ such that
$$ \parallel f^{\prime} (x_{m_{k}})(h_{k}) \parallel \geq \sup_{m}\parallel f^{\prime} (x_{m})(h_{k}) \parallel-\frac{1}{k} ~~~~(k\in \mathbb{N}).$$
Since the sequences $ (x_{m_{k}}) _{k}$ in $ U $ and $ (h_{1},0,h_{2},0,h_{3},0,\cdot\cdot\cdot) $ in $ X $
are weakly $ p $-Cauchy, the sequence
$$ (f^{\prime}(x_{m_{1}})(h_{1}),0,f^{\prime}(x_{m_{2}})(h_{2}),0, f^{\prime}(x_{m_{3}})(h_{3}),0,\cdot\cdot\cdot) $$
converges in $ Y. $\ Therefore, $ \displaystyle\lim_{k} f^{\prime}(x_{m_{k}})(h_{k})=0.$\ Hence, $ \rm{(ii)} $ implies that
$$ \lim_{k} \sup_{m}\parallel f^{\prime}(x_{m}) (h_{k}) \parallel =0.$$
(iii) $ \Rightarrow $ (iv) is obvious.\\
(iv) $ \Rightarrow $ (v) Let $ K$ be a weakly $ p $-precompact $ U $-bounded set.\ It is clear that, for every $ h \in X, $ the set $ f^{\prime}(K)(h) $
is bounded
in $ Y. $\ Let $ (h_{n}) _{n}$ is a weakly $ p $-summable sequence in $ X. $\
If $ (h_{n_{k}})_{k} $ is a subsequence of $ (h_{n}) _{n},$ then for every $ k\in \mathbb{N}, $ there exists $ a_{k}\in K $ such that
$$ \sup_{a \in K}\parallel f^{\prime}(a) (h_{n_{k}}) \parallel < \parallel f^{\prime}(a_{k}) (h_{n_{k}}) \parallel+\frac{1}{k}. $$
Since $ K $ is a weakly $ p $-precompact set, the sequence $ (a_{k})_{k} $ admits a weakly $ p $-Cauchy subsequence $ (a_{k_{r}}) _{r}.$\ Hence, by the hypothesis, 
$$ \displaystyle\lim_{r} \parallel f^{\prime}(a_{k_{r}}) (h_{n_{k_{r}}}) \parallel=0. $$
Therefore we have,
$ \displaystyle\lim_{r}\sup_{a\in K} \parallel f^{\prime}(a) (h_{n_{k_{r}}}) \parallel=0. $\ Hence, every subsequence of
$ ( \displaystyle \sup_{a \in K} \parallel f^{\prime}(a) (h_{n}) \parallel) _{n}$
has a subsequence converging to $ 0. $\ Therefore, the sequence itself converges to $ 0, $ that is, $ \displaystyle \lim_{n}\sup_{a \in K} \parallel f^{\prime}(a) (h_{n}) \parallel =0.$\\
(v) $ \Rightarrow $ (i) Let $ (x_{n})_{n} $ be a $ U $-bounded and weakly $ p $-Cauchy sequence.\ Since $ U$ is convex, the segment $ I(x_{n}, x_{m}) $ is contained in $ U$
for all $ n,m \in \mathbb{N}. $\ By the Mean Value Theorem {\rm (\cite[Theorem 6.4]{ch})},
there exists $ c_{nm} \in I(x_{n}, x_{m}) $ such that
$$ \parallel f(x_{n}) -f(x_{m}) \parallel \leq \parallel f^{\prime}(c_{n,m}) (x_{n}-x_{m}) \parallel \leq \sup_{i,j \in \mathbb{N}}\parallel f^{\prime}(c_{i,j}) (x_{n}-x_{m}) \parallel .$$
Since $ (c_{i,j})_{i,j} $ is a weakly $ p $-Cauchy and $ U $-bounded sequence, by $ \rm{(v)} $ we have
$$ \lim_{n,m}\sup_{i,j \in \mathbb{N}}\parallel f^{\prime}(c_{i,j}) (x_{n}-x_{m}) \parallel=0. $$
Therefore, $\displaystyle \lim_{n,m\rightarrow\infty}\parallel f(x_{n}) -f(x_{m}) \parallel =0.$
\end{proof}
\begin{rem}\label{r2}
If $ U \subseteq X $ is an open subset, let $ f \in C^{1u}(U, Y ) \cap C^{p}_{wsc}(U, Y ). $ \ Then, by $ \rm{(ii)} $ of Theorem $ \rm{\ref{t1}} $, $ f^{\prime}(U) \subseteq C_{p}(X, Y ). $\
\end{rem}
\begin{thm}\label{t2}
Let $ B_{X} $ be weakly $ p $-precompact and let $ U $ be an open subset of $ X$ and let $ f \in C^{1u}(U, \mathbb{R} )\cap C^{p}_{wsc}(U, \mathbb{R} ). $\ If every uniformly $ p $-convergent set in $ K(X, \mathbb{R} ) $ is uniformly completely continuous, then $ f^{\prime} $ is compact with values in $ C_{p}(X, \mathbb{R} ) .$
\end{thm}
\begin{proof}
By the Remark \ref{r2}, the range of $ f^{\prime} $ is contained in $ C_{p}(X, \mathbb{R} ) = K(X, \mathbb{R} ) $ (see Theorem 2.6 in \cite{ccl}), where we have used
the fact that $ B_{X} $ is weakly $ p $-precompact. Let $ B$ be a $ U $-bounded set.\ Given $ h \in X, $ let $ (x_{n})_{n} \subset B $ be a sequence.\ Since $ B_{X} $ is weakly $ p $-precompact, we can assume that $ (x_{n})_{n} $ is weakly $ p $-Cauchy.\ Hence Theorem \ref{t1}$ \rm{(ii)} $, implies that the sequence $ ( f^{\prime} (x_{n})(h))_{n} $ is converges
in $ \mathbb{R} $ and so, $ f^{\prime} (B)(h) $ is relatively compact in $ \mathbb{R}.$\ Moreover, by the  Theorem \ref{t1}$ \rm{(v)} $, $ f^{\prime} (B) $ is a uniformly $ p $-convergent set in $ K(X, \mathbb{R} ) .$\ Hence by the hypothesis $ f^{\prime} (B) $is 
uniformly completely continuous
in $ K(X, \mathbb{R} ) .$\ On the other hand by using again the fact that $ B_{X} $ is weakly $ p $-precompact, we observation that
$ X $ contains no copy of $ \ell_{1} .$\ Therefore, {\rm (\cite[Theorem 1]{m})} implies that
$ f^{\prime} (B) $ is relatively
compact in $ C_{p}(X, \mathbb{R} ) .$
\end{proof}

\begin{defn}\label{d3}
A mapping $ f:U \rightarrow Y $ is $ p $-convergent, if for every $ U $-bounded and weakly $ p $-converging sequence $ (x_{n}) _{n}$ to $ x\in U $, the sequence
$ ( f (x_{n}))_{n} $ norm converges to $ f (x) .$\ We denote the space of $ p $-convergent mappings by $ C_{C_{p}}(U, Y ) .$\
It is clear that for all $ X, Y , $ and for every open subset $ U \subseteq X, $ we have $ C_{wsc}^{p} (U,Y)\subseteq C_{C_{p}}(U, Y ) .$\\
We do not have any example of a mapping $ f\in C^{1u}(U, Y ) \cap C_{C_{p}}(U, Y ) $ which does not belong to $ C_{wsc}^{p} (U,Y). $
\end{defn}
\begin{thm}\label{t3}
Let $ U\subseteq X $ be an open convex subset and let $ f \in C^{1u}(U, Y ). $\ If $ f^{\prime} $ takes $ U$-bounded weakly $ p $-compact subsets of $ U$ into relatively compact subsets of $ K(X, Y), $ then the following assertions are hold:\\
$ \rm{(i)} $ $ f \in C_{C_{p}}(U, Y ); $\\
$ \rm{(ii)} $ $ f^{\prime} \in C_{C_{p}}(U,K(X, Y )). $
\end{thm}
\begin{proof}
$ \rm{(i)}$ Let $ (x_{n})_{n} $ be a $ U$-bounded sequence and weakly $ p $-convergent to $ x \in U. $\ Since the set $ K:=\lbrace x_{n} :n\in \mathbb{N} \rbrace \cup\lbrace x\rbrace $ is weakly $ p $-compact subset of $ U,$ by the Mean Value Theorem {\rm (\cite[Theorem 6.4]{ch})},
for each $ n \in \mathbb{N} , $ there is $ c_{n} \in I(x_{n}, x) $ such that
$$ \parallel f(x_{n}) -f(x) \parallel \leq \parallel f^{\prime}(c_{n}) (x_{n}-x) \parallel \leq \sup_{k}\parallel f^{\prime}(c_{k}) (x_{n}-x)\parallel ~~~~~~(n\in \mathbb{N}).$$
But, $ (c_{n})_{n} $ is a weakly $ p $-convergent sequence to $ x$ and $ U $-bounded.\ Therefore, by the hypothesis, the set $ \lbrace f^{\prime}(c_{n}) :n\in \mathbb{N} \rbrace $ is relatively compact in $ K(X, Y ) $
and so, uniformly $ p $-convergent.\ Since $ (x_{n}- x)_{n} $ is weakly $ p $-summable, we have:
$$ \displaystyle\lim_{n} \parallel f(x_{n})-f(x) \parallel\leq \lim_{n} \sup_{k}\parallel f^{\prime}(c_{k}) (x_{n}-x)\parallel =0.$$
$\rm{(ii)} $ Suppose that $ (x_{n})_{n} $ is a $ U $-bounded and weakly $ p $-convergent sequence.\ Consider $ B := \lbrace x_{n}: n \in \mathbb{N} \rbrace\cup\lbrace x\rbrace $ and let $ d := min\lbrace 1, dist(B, \partial U)\rbrace. $\ The set
$$ B^{\prime}:=B+\frac{d}{2}B_{X}\subset U$$
is also $ U $-bounded.\ Since $ f \in C^{1u}(U, Y ),$ $ f^{\prime} $ is uniformly continuous on $ B^{\prime}.$\ Hence, for given $ \varepsilon >0, $ there exists $ 0 <\delta <\frac{d}{4} $ such that if $t_{1} ,t_{2}\in B^{\prime}$
satisfy $ \parallel t_{1} -t_{2}\parallel <2\delta,$ then
\begin{equation}\label{eq1}
\parallel f^{\prime}(t_{1})-f^{\prime}(t_{2}) \parallel <\frac{\varepsilon}{4}.
\end{equation}
Fix $ n\in\mathbb{N} $ and $ h\in B_{X}. $\
If $ c\in I(x_{n},x_{n}+ x_{n}h) ,$ then as in the proof of Theorem \ref{t1} we have
$$ \parallel c-x_{n} \parallel \leq <2\delta <\frac{d}{2},$$
and so,
$c=x_{n} +(c-x_{n})\in B^{\prime}=B+\frac{d}{2}B_{X}. $\ Therefore, $ I(x_{n}, x_{n} +\delta h)\subset B^{\prime}. $\ Similarly, $ I(x, x+\delta h) \subset B^{\prime}.$\
As an immediate consequence of the Mean Value Theorem {\rm (\cite[Theorem 6.4]{ch})},
and formula (\ref{eq1}), we obtain
\begin{flushleft}
$ \parallel f^{\prime}(x_{n}) ( \delta h)-f(x_{n}+
\delta h)+f(x_{n})\parallel ~ ~~~`$
\end{flushleft}
$$ ~~~~~~~\leq \displaystyle\sup_{c\in I(x_{n},x_{n}+\delta h)} \parallel f^{\prime} (c) - f^{\prime} (x_{n}) \parallel \parallel\delta h\parallel\leq\frac{\varepsilon \delta}{4} . $$
Similary,
\begin{flushleft}
$ \parallel f(x+
\delta h)-f(x)- f^{\prime}(x) ( \delta h)\parallel ~ $
\end{flushleft}
$$ \leq \displaystyle\sup_{c\in I(x,x+\delta h)} \parallel f^{\prime} (c) - f^{\prime} (x) \parallel \parallel\delta h\parallel\leq\frac{\varepsilon \delta}{4} . $$
In the other word, the sequences $ (x_{n}+\delta h)_{n} $ and $ (x_{n})_{n} $ are $ U$-bounded and weakly $ p $-convergent in $ U $ to $ x + \delta h $ and $ x, $ respectively.\
 So, there exists $ n_{0}\in \mathbb{N} $ so that for $ n> n_{0},$\ 
$ \parallel f(x_{n}+\delta h)-f(x+\delta h)\parallel <\frac{\varepsilon \delta}{4},~~~$  and $~~~~ \parallel f(x_{n})-f(x)\parallel <\frac{\varepsilon \delta}{4}. $

Therefore, for $ n> n_{0} ,$
$\parallel f^{\prime}(x_{n}) (h) - f^{\prime}(x) (h)\parallel< \varepsilon.$\
Hence, for every $ h \in X, $ we have $ \displaystyle \lim_{n}f^{\prime}(x_{n})(h)=f^{\prime}(x)(h).$\ Since $ f^{\prime}(B) $ is relatively compact in $ K(X, Y ), $  Lemma 4.2 of \cite{va}, implies that
$ \displaystyle \lim_{n}f^{\prime}(x_{n})=f^{\prime}(x).$\
\end{proof}


\end{document}